\documentclass[12pt]{amsart}
\usepackage{amsfonts,amssymb,amscd,amsmath,enumerate,verbatim,calc,amsthm, mathrsfs}
\usepackage[colorlinks,linkcolor=blue,pagebackref=true,citecolor=red]{hyperref}
\usepackage[mathscr]{euscript} 
\usepackage{color}
\usepackage{graphicx}
\usepackage[all]{xy}
\usepackage[top=1in, bottom=1.3in, left=1in, right=1in]{geometry}
\def\Proj{\operatorname{Proj}}  
\def\Spec{\operatorname{Spec}}
 
\def\rdim{\operatorname{rdim}}

\def\thebibliography#1{\section*{References}
	\list{[\arabic{enumi}]}{\settowidth \labelwidth{[#1]} \leftmargin
		\labelwidth \advance \leftmargin \labelsep \usecounter{enumi}}
	\def\newblock{\hskip .11em plus .33em minus .07em} \sloppy
	\clubpenalty 4000 \widowpenalty 4000 \sfcode`\.=1000 \relax}

\def\frk{\mathfrak}

\def\Phi{{\frk N}}
\def\opn#1#2{\def#1{\operatorname{#2}}} 
\opn\chara{char} \opn\length{\ell} \opn\pd{pd} \opn\rk{rk}
\opn\projdim{proj\,dim} \opn\injdim{inj\,dim}
\opn\rank{rank} \opn\depth{depth} \opn\grade{grade} 
\opn\hei{ht} \opn\embdim{emb\,dim}\opn\codim{codim}
\opn\Tr{Tr} \opn\bigrank{big\,rank}
\opn\superheight{superheight} \opn\lcm{lcm}
\opn\rdim{rdim} \opn\trdeg{tr\,deg} \opn\reg{reg}  \opn\lreg{lreg} 
\opn\ini{in} \opn\lpd{lpd} \opn\size{size} \opn{\mult}{mult}
\opn\div{div} \opn\Div{Div} \opn\cl{cl} \opn\Cl{Cl}
\opn\Spec{Spec} \opn\Supp{Supp} \opn\supp{supp} 
\opn\Sing{Sing} \opn\Ass{Ass} \opn\Min{Min}
\opn\Proj{Proj}
\opn\Ann{Ann} \opn\Rad{Rad} \opn\Soc{Soc}
\opn\Syz{Syz} \opn\Im{Im} \opn\Ker{Ker} \opn\Coker{Coker}
\opn\Am{Am} \opn\Hom{Hom} \opn\tor{Tor} \opn\Ext{Ext}
\opn\End{End} \opn\Aut{Aut} \opn\id{id}

\opn\nat{nat} \opn\pff{pf} 
\opn\Pf{Pf} \opn\GL{GL} \opn\SL{SL} \opn\mod{mod} \opn\ord{ord}
\opn\Gin{Gin} \opn\Hilb{Hilb}
\opn\adeg{adeg} \opn\std{std}\opn\ip{infpt}
\opn\Pol{Pol} \opn\sat{sat} \opn\Var{Var}
\opn\aff{aff} \opn\con{conv} \opn\relint{relint} \opn\st{st}
\opn\lk{lk} \opn\cn{cn} \opn\core{core} \opn\vol{vol}
\opn\link{link} \opn\star{star}
\opn\gr{gr}

\def\pot#1#2{#1[\kern-0.28ex[#2]\kern-0.28ex]}

\opn\dirlim{\underrightarrow{\lim}}
\opn\inivlim{\underleftarrow{\lim}}

\newtheorem{Theorem}{Theorem}[section]
\newtheorem{Lemma}[Theorem]{Lemma}
\newtheorem{Corollary}[Theorem]{Corollary}
\newtheorem{Proposition}[Theorem]{Proposition}
\newtheorem{Remark}[Theorem]{Remark}

\newtheorem{Example}[Theorem]{Example}

\newtheorem{Definition}[Theorem]{Definition}

\newtheorem{Conjecture}[Theorem]{Conjecture}

\let\phi=\varphi
\let\kappa=\varkappa

\opn\dis{dis}
\opn\Lex{Lex}

\opn{\MinAss}{MinAss}

\begin{document}
	
\title[]{Minkowski inequality and equality for multiplicity of ideals of finite length in Noetherian local rings}

\author[]{Kriti Goel}
\address{Indian Institute of Technology Bombay, Mumbai, INDIA 400076}
\email{kriti@math.iitb.ac.in}
\thanks{The first author is supported by a UGC fellowship, Govt. of India}

\author[]{R. V. Gurjar}
\address{Indian Institute of Technology Bombay, Mumbai, INDIA 400076}
\email{gurjar@math.iitb.ac.in} 

\author[]{J. K. Verma}
\address{Indian Institute of Technology Bombay, Mumbai, INDIA 400076}
\email{jkv@math.iitb.ac.in}

\thanks{{\it 2010 AMS Mathematics Subject Classification:}Primary 13-02, 13B22, 13F30, 13H15, 14C17}

\thanks{{\it Key words}: multiplicity and mixed multiplicities  of ideals, integral closure of ideals, geometric interpretation of multiplicity, $\mathfrak m$-valuations.}

\begin{abstract}
In this exposition of the equality and inequality of Minkowski for multiplicity of ideals, we provide simple algebraic and geometric proofs. Connections with mixed multiplicities of ideals are explained. 
\end{abstract}

\maketitle

\section{Introduction}
The objective of this expository paper is to present an account of Minkowski's inequality and equality for multiplicity of ideals of finite co-length in Noetherian local rings. These were first investigated by B. Teissier in his Carg\`ese paper \cite{teissier1973}, in which he proposed conjectures about mixed multiplicities of ideals which imply Minkowski's inequality for multiplicities of ideals. We shall present  geometric proofs using an interpretation of multiplicity of an ideal due to C. P. Ramanujam \cite{CPR}.
We present a simpler version of the Rees-Sharp  proof of Minkowski equality for multiplicities in dimension 2. A proof  of Minkowski's equality in dimension $\geq 3$ is presented using specialization of the integral closure of ideals due to S. Itoh \cite{itoh1992} and Hong-Ulrich \cite{hongUlrich}.
In his book {\em ``Singular points of hypersurfaces"} \cite{milnor1968}, John Milnor proved the following:

\begin{Theorem}
	Suppose  that the origin is an isolated singular point of a complex analytic hypersurface $H=V(f)\subset \mathbb{C}^{n+1}$ and  $S$ is an $n$-dimensional  sphere centered at the origin of sufficiently small radius. Define $\varphi(z): S\setminus H \to S^1 \text{  given by } \varphi(z)=\frac{f(z)}{||f(z)||}.$ Then each fiber of $\varphi$ upto homotopy is a wedge of $\mu$ copies of $S^n$ where 
	\[\mu=\dim_{\mathbb{C}}\frac{\mathbb{C}\{z_0, z_1, \ldots, z_n\}}{(f_{z_0}, f_{z_1},\ldots, f_{z_n})}.\]
\end{Theorem}

The number $\mu$ is called the Milnor number of the hypersurface $H$ at the origin. The Milnor number of $X$ at an isolated singular point $x\in X$ will be denoted by $\mu(X,x).$ B. Teissier \cite{teissier1973} showed that $\mu(X,x)$ is a topological invariant.
\begin{Theorem}[\bf Teissier, 1973]
	Let $(X, x)$ and $(Y, y)$ be two germs of hyper surfaces with isolated singularity having  same topological type. Then  \[\mu(X,x)=\mu(Y,y).\]
\end{Theorem}
 
Teissier, in his Carg\`ese paper \cite{teissier1973}, refined the notion of Milnor number. He replaced it by a sequence of Milnor numbers of intersections of $X$ with general  linear subspaces.

\begin{Theorem}[\bf Teissier, 1973]
	Let $(X, x)$ be a germ of a hypersurface in  $\mathbb{C}^{n+1}$ with an isolated singularity. Let $E$ be an $i$-dimensional affine subspace of $\mathbb{C}^{n+1}$ passing through $x.$ If $E$ is sufficiently general, then the Milnor number of $X \cap E$ at $x$ is independent of $E.$ 
\end{Theorem}

\begin{Definition} 
	The Milnor number of $X\cap E$, where $E$ is a general linear subspace of dimension $i$ passing through $x$ is called the $i^{th}$-sectional Milnor number of $X.$ It is denoted by $\mu^{(i)}(X, x).$  These are collected together as
	\[\mu^*(X, x) = (\mu^{(n+1)}(X, x), \mu^{(n)} (X, x), \ldots , \mu^{(0)}(X, x)).\]
\end{Definition}

It is easy to see that $\mu^{(0)}(X, x)=1$ and $\mu^{(1)}(X,x)=m_x(X)-1$ where $m_x(X)$ denotes the multiplicity of $X$ at $x.$ Teissier proposed the following: 

\begin{Conjecture}\label{mu}
	If the germs of isolated hypersurface singularities $(X,x)$ and $(Y,y)$ have same topological type then  \[\mu^*(X,x)=\mu^*(Y,y).\]
\end{Conjecture}

This conjecture contains the following conjecture of O. Zariski \cite{zariski1971}.

\begin{Conjecture}[\bf Zariski, 1971]
	For topologically equivalent isolated singularities of hypersurfaces, $m_x(X)=m_y(Y).$ 
\end{Conjecture}

Zariski conjecture has been established by L\^{e} D\~{u}ng Tr\'{a}ng for plane curves in~\cite{trang1973}. Teissier's Conjecture \ref{mu} was disproved in 1975 by J. Brian\c{c}on and J.-P. Speder \cite{brionconskoda1975}. Another conjecture made by Teissier was about the log-convexity of the sectional Milnor numbers:

\begin{Conjecture}[\bf Teissier, 1973] 
	Is it true that
	\[\frac{\mu^{(n+1)}}{\mu^{(n)}}\geq  \frac{\mu^{(n)}}{\mu^{(n-1)}}\geq \cdots \geq \frac{\mu^{(1)}}{\mu^{(0)}}. \]
\end{Conjecture}

Since all the sectional Milnor numbers are positive integers, these inequalities are equivalent to asking if for all $i=1,2,\ldots, n,$
\[\log \mu^{(i)}\leq \frac{1}{2}\left(\log \mu^{(i-1)}+\log \mu^{(i+1)}\right).\]
In other words the sequence $\mu^{(0)}, \mu^{(1)},\ldots, \mu^{(n+1)}$ is a log-convex sequence.
In an appendix to the paper by David Eisenbud and Harold I. Levine \cite{teissierAppendix}, Teissier answered this in the affirmative in a much more general setting. Following a suggestion of Hironaka, he considered the Bhattacharya function \cite{bhattacharya1957} of two ideals to identify the sectional Milnor numbers with  mixed multiplicities of ideals. First we recall these notions and related results.

\begin{Theorem} 
	Let $(R, \mathfrak{m})$ be a  Noetherian local ring of dimension $d$ and let $I$ be an  $\mathfrak{m}$-primary ideal. The Hilbert function of $I,$ $H_I(n):=\ell (R/I^n)$ is given by the Hilbert polynomial of $I,$
	\[P_I(x)=e_0(I)\binom{x+d-1}{d}-e_1(I)\binom{x+d-2}{d-1}+\cdots+(-1)^de_d(I)\]
	for all large $n.$ The coefficients $e_0(I), e_1(I),\ldots, e_d(I)$ are integers.  The leading coefficient  $e_0(I)$ is a called the multiplicity of $I$ and is denoted by $e(I).$ 
\end{Theorem}

The above result of P. Samuel was extended for two $\mathfrak{m}$-primary ideals by Phani Bhushan Bhattacharya in 1957 \cite{bhattacharya1957}.

\begin{Theorem} [\bf P. B. Bhattacharya] 
	Let $I$ and $J$ be $\mathfrak{m}$-primary ideals of a $d$-dimensional Noetherian local ring $(R,\mathfrak{m}).$  Then the function
	\[H(r,s)=\ell\left(\frac{R}{I^rJ^s}\right)\]
	is given by a polynomial $P(r,s)$ of degree $d,$ when $r$ and $s$ are large.
\end{Theorem} 

The polynomial $P(r,s),$   called the {\em Bhattacharya polynomial} of $I$ and $J$, can be written as
\[P(r,s)=\frac{1}{d!}\sum_{i=0}^d\binom{d}{i}e_i(I|J)r^{d-i}s^i+\text{ terms of degree } < d \]
where $e_0(I|J), e_1(I|J),\ldots, e_d(I|J)$ are called the {\em mixed multiplicities} of $I$ and $J.$ David Rees showed that  $e_0(I|J)=e(I)$ and  $e_d(I|J)=e(J).$ The other mixed multiplicities are also multiplicities of certain system of parameters \cite{teissier1973}. 

\begin{Theorem}[\bf J. J. Risler and Teissier] 
	Each  mixed multiplicity $e_i(I|J)$ is the multiplicity of an ideal generated by $(d-i)$ general elements from $I$ and $i$ general elements from $J.$
\end{Theorem}
 
Risler and Teissier developed the theory of mixed multiplicities for a finite family of $\mathfrak{m}$-primary ideals in a local ring $(R,\mathfrak{m}).$ Since we need this theory for only two ideals, we shall consider it only in this case for the sake of simplicity.

Rees introduced joint reductions \cite{rees1984} in order to find systems of parameters whose multiplicities are the mixed multiplicities. Recall that an ideal $J\subset I$ is called a reduction of $I$ if there is an $n$ such that $JI^n=I^{n+1}.$ Let $I_1, I_2,\ldots, I_d$ be $\mathfrak{m}$-primary ideals of $R.$ A sequence of elements $x_1\in I_1, \ldots, x_d\in I_d$ is called a joint reduction of the sequence of ideals  $(I_1,\ldots, I_d)$ if $\sum_{i=1}^dx_iI_1I_2\cdots I_{i-1}I_{i+1}\cdots I_d$ is a reduction of $I_1I_2\cdots I_d.$ Let $(I^{[m]}, J^{[n]})$ denote the sequence of ideals in which the first $m$ ideals are $I$ and the next $n$ ideals are $J.$

\begin{Theorem}[\bf Rees, 1984] 
	Let $x_1, x_2,\ldots, x_d$ be a joint reduction of the sequence of ideals $(I^{[d-i]}, J^{[i]})$ then $e_i(I|J)=e(x_1, x_2,\ldots, x_d).$
\end{Theorem}

\begin{Theorem} [\bf Teissier, 1973]
	Let $X=V(f)$ be  an analytic hypersurface in  $\mathbb{C}^{n+1}$ with an isolated singularity at the origin. Then for all $i=0,1,\ldots, n+1, $ 
	\[\mu^{(i)}(X,0)= e_{i}(\mathfrak{m}|J(f)).\]
\end{Theorem}

This result led Teissier to propose the following conjecture, which if true, implies that the sequence $\mu^*(X,x)$ is log-convex.

\begin{Conjecture}[\bf Teissier's Second Conjecture]
	Let $I$ and $J$ be $\mathfrak{m}$-primary ideals of a Noetherian local ring of dimension $d\geq 2.$ Put $e_i=e_i(I|J).$ Then
	\[ \frac{e_1}{e_0}\leq \frac{e_2}{e_1}\leq \frac{e_3}{e_2}\leq \cdots\leq \frac{e_{d}}{e_{d-1}}.\]
\end{Conjecture}
 
Using the Bhattacharya polynomial, we can see that for all  $r,s\in \mathbb{N}$,
\[e(I^rJ^s)=e(I)r^d+\binom{d}{1}e_1(I|J)r^{d-1}s+\cdots+e_i(I|J)\binom{d}{i}r^{d-i}s^i+ \cdots+ e(J)s^d.\]
 \[e(IJ)=e(I)+\binom{d}{1}e_1(I|J)+\cdots+e_i(I|J)\binom{d}{i}+ \cdots+ e(J).\]

Teissier compared the expansion 
\[(e(I)^{1/d}+e(J)^{1/d})^d=e(I)+\cdots+ \binom{d}{i}e(I)^{(d-i)/d}e(J)^{i/d}+\cdots + e(J)\]
with the formula for $e(IJ)$ and proposed his first conjecture:

\begin{Conjecture}[\bf Teissier's First Conjecture] 
	Let $(R,\mathfrak{m})$ be a Noetherian local ring of dimension $d$ and $I, J$ be $\mathfrak{m}$-primary ideals. Then  for all $i=0,1,\ldots, d,$
	\[e_i(I|J)^d\leq e(I)^{d-i}e(J)^i.\]
\end{Conjecture}

Teissier asked if the Minkowski type inequality is true for the multiplicity of ideals:
\[ e(IJ)^{1/d} \leq e(I)^{1/d} + e(J)^{1/d}.\]
 
Henceforth, we call this Minkowski's inequality. Note that the second conjecture implies the first and the first conjecture implies that Minkowski's inequality is true for multiplicities of ideals.
He  proved Minkowski's inequality for multiplicities in \cite{teissier1978}   for  Cohen-Macaulay local algebras  over an algebraically closed field of characteristic zero in dimension two and then used superficial elements to prove it in any dimension. D. Rees and R. Y. Sharp \cite{reesSharp1978} proved it in all Noetherian local rings.
 
\begin{Theorem}[\bf Rees-Sharp, 1978]  
	The first and the second conjecture of Teissier are true for all Noetherian local rings. In particular, Minkowski's inequality  for multiplicity is true in all  Noetherian local rings.
\end{Theorem} 

It is natural to ask under what conditions Minkowski inequality is an equality. This requires the concept of integral closure of an ideal, which we recall. Let $I$ be an ideal of a commutative ring $R.$  We say that $x\in R$ is {\em integral over $I$}  if there exist $a_i\in I^i$ for $i=1,\ldots, n$ such that $x^n+a_1x^{n-1}+\cdots+a_n=0.$ The integral closure of $I$ is the ideal 
\[\overline{I}=\{ x\in R \mid x \text{ is integral over } I\}.\]
Recall that a Noetherian local ring $(R,\mathfrak{m})$ is called {\em quasi-unmixed}  if for all minimal primes $\mathfrak{p}$ of the $\mathfrak{m}$-adic completion $\hat{R},$ we have $\dim \hat{R}/\mathfrak{p}=\dim R.$

\begin{Theorem}
	Let $(R,\mathfrak{m})$ be quasi-unmixed and $d\geq 2.$ Then
	\begin{align*}
	e(IJ)^{1/d} = e(I)^{1/d} + e(J)^{1/d} 
	&\iff \overline{I^r}=\overline{J^s} \text{ for some } r,s\in \mathbb{N} \\
	&\iff \frac{e_i}{e_{i-1}}=\frac{r}{s} \; \text{ for all } i=1,\ldots,d.
	\end{align*}
\end{Theorem}

It was  proved by Teissier for complex analytic Cohen-Macaulay algebras in arbitrary dimension, by reduction to the case of dimension 2, which appeared in the  1978 Conference Proceedings dedicated to C. P. Ramanujam \cite{teissier1978}. Rees and Sharp proved it for $d=2$ and D. Katz reduced the proof for  dimension $\geq 3$ to dimension $2$ in \cite{katz1988}.

\medskip

June Huh \cite{huh} settled a long standing conjecture in graph theory about chromatic polynomials using sectional Milnor numbers and mixed multiplicities. Some of the  results he proved are:
(1) If $J$ is an ideal of a standard graded domain over an algebraically closed field generated by elements of the same degree, then the mixed multiplicities of $\mathfrak{m}$ and $J$ form a log-concave sequence of nonnegative integers with no internal zeros. (2) Let $h\in  \mathbb{C}[x_0, x_1,\ldots, x_n]$ be a homogeneous  polynomial of positive degree. Put
\[V(h)=\{p\in \mathbb{P}^n\mid h(p)=0\} \text{ and } D(h)=\mathbb{P}^n\setminus V(h).\] 
Let $\mu^i(h)$ be the $i^{th}$ mixed multiplicity of $\mathfrak{m}$ and $J(h).$ Then the Euler Characteristic of $D(h)$ is given by
\[\chi(D(h))=\mu^0(h)-\mu^1(h)+\cdots+(-1)^n\mu^n(h).\]
(3) The numbers $\mu^i(h)$ form a log-concave sequence of nonnegative integers with no internal zeros for any $h.$ He uses above results and various properties of matroids, mixed multiplicities, mixed volumes of convex bodies and Milnor numbers to show that the coefficients of the chromatic polynomial of a graph form a log-concave sequence.

\medskip

Recently, Minkowski's inequality has been proved for non-Noetherian filtrations of ideals by  Dale Cutkosky, Parangama Sarkar and Hema Srinivasan \cite{css}. Multiplicity of non-Noetherian filtrations of ideals has been investigated by many authors. The most general result for such filtrations was proved by Cutkosky \cite{cutcosky2004}. Let $N(\hat{R})$ denote the nilradical of $\hat{R}.$

\begin{Theorem}[\bf Cutkosky, 2004]
	Let $\mathcal F=\{I_n\}$ be a filtration of $\mathfrak{m}$-primary ideals of a $d$-dimensional Noetherian local ring $R.$ Then the limit 
	\[\lim_{n\to \infty}\frac{\ell(R/I_n)}{n^d} \text{ exists } \iff \dim N(\hat{R})< d.\]
\end{Theorem}

This limit has been  investigated by several authors, for example, by Ein, Lazarsfeld and Smith \cite{els2003} and Must\c at\u a \cite{mus2002}.

\begin{Theorem}[\bf Cutkosky-Sarkar-Srinivasan, 2018] 
	Let $(R,\mathfrak{m})$ be a $d$-dimensional Noetherian local ring and let $\{I(1)_n\},\ldots, \{I(r)_n\}$ be filtrations of $\mathfrak{m}$-primary ideals.   Let $M$ be a finitely generated $R$-module.  If $\dim N(\hat R)< d$, then the function 
	\[P(n_1, n_2,\ldots, n_r)=\lim_{m\to \infty}\frac{\ell(M/I(1)_{mn_1}\cdots I(r)_{mn_r}M)}{m^d}\]
	is a real homogeneous polynomial $G(n_1,\ldots, n_r)$ of degree $d$ for all $n_1,\ldots, n_r\in \mathbb{N}.$ and it can be written as
	\[G(n_1, n_2,\ldots, n_r)=\sum_{d_1+d_2+\cdots+d_r=d}
	e_R(I(1)^{[d_1]},\ldots , I(r)^{[d_r]}; M)\frac{n_1^{d_1}\cdots n_r^{d_r}}{d_1!\cdots d_r!}.\]
\end{Theorem}

The integer $e_R(I(1)^{[d_1]},\ldots, I(r)^{[d_r]}; M)$ is called the mixed multiplicity of the $\mathbb{N}^r$-graded filtration $\{I(1)_{n_1}\cdots I(r)_{n_r} \mid n_1,\ldots, n_r\in \mathbb{N}\}$  of the type $(d_1,\ldots, d_r).$ Many of the classical properties of mixed multiplicities of ideals continue to be true in this setting. In particular, Minkowski's inequalities are proved in \cite{css}.\\

Several  historical facts mentioned in this paper are taken from the presentation of Patrick Popescu-Pampu made in the conference {\em Singular Landscapes,} held in  honour of Teissier's 70th birthday in 2015 \cite{pampu}. \\

{\bf Acknowledgements:} Thanks are due to D. Katz and B. Teissier for a careful reading of the manuscript and for providing detailed comments which improved the exposition. The third author thanks Aldo Conca and Marilina Rossi for inviting him to Genoa to offer a course on Hilbert functions. Some of the topics covered in this paper were presented in this course. Financial support from  Istituto Nazionale di Alta Matematica made this visit possible.

\section{C.P. Ramanujam's result, Teissier's Conjecture and a related inequality}

In this section, we use basic intersection theory of divisors on smooth algebraic varieties to give short proofs of three results: (1) C.P. Ramanujam's geometric interpretation of multiplicity, (2) Minkowski's inequality for multiplicity, and (3) necessary and sufficient conditions for Minkowski's equality. B. Teissier gave geometric proofs of (2) and (3) in \cite{teissier1978}. We believe that the  proofs presented here are more accessible to a young reader.\\

For the sake of exposition, we will assume that $k$ is an algebraically closed field. We will only consider local rings whose residue field is isomorphic to $k$. By an algebraic local ring we mean either a local ring of an algebraic variety over $k$ at a maximal ideal, or a local analytic ring of the form $\mathbb{C}\{Z_1,Z_2,\ldots,Z_n\}/P$, where $\mathbb{C}\{Z_1,\ldots,Z_n\}$ is the convergent power series ring over the complex field and $P$ is a prime ideal. We assume that the reader is familiar with the basic notions in algebraic geometry.\\

{\bf Intersection Theory.}
Our proofs of the results depend crucially on the intersection theory of curves on a smooth surface, or divisors in higher dimensional smooth varieties. We will recall this basic theory, mostly without proofs. An excellent source for the detailed properties of intersection theory in arbitrary dimension is the book \cite{F}.\\

Let $X$ be a smooth irreducible surface over $k$. Let $C,D$ be (possibly non-reduced and reducible) curves on $X$ without a common irreducible component. For a point $p\in C\cap D$, there are functions $r,s$ in the local ring $R$ of $X$ at $p$ such that $C,D$ are scheme-theoretically defined by $r,s$ respectively. Then $\dim_k R/(r,s)$ is called the \emph{intersection multiplicity} of $C$ and $D$ at $p$, denoted by $i(C,D;p)$. It can be proved that if $C_m,D_n$ are all the irreducible components of $C,D$ respectively passing through $p$ (counting multiplicities), then $i(C,D;p)=\Sigma_{m,n} \ i(C_m,D_n;p)$. A useful result in this connection is the following lemma. Recall that the residue field $k$ of local rings occuring in the result below is assumed to be algebraically closed.

\begin{Lemma} [Abhyankar, \cite{abhyankar}]
	Let $(R,M)$ be a Cohen Macaulay $1$-dimensional algebraic or local analytic ring, and let $\overline{R}$ be the normalization of $R$ in its total quotient ring. Then for any nonzerodivisor $r\in R$, $\dim_k R/rR=$ $\dim_k\overline R/r\overline R$.
\end{Lemma}

\begin{proof} 
	Let $\overline R/r\overline R \supset I_1\supset \ldots\supset I_l=(0)$ be a Jordan-H{\"o}lder sequence of ideals. Then each $I_j/I_{j+1}$ is a $1$-dimensional $k$-vector space and hence an $R/M$-module of length $1$. Therefore, this is also a Jordan-H{\"o}lder sequence of $R$-modules. This implies
	\[l_R(\overline R/r\overline R)=l_{\overline R}(\overline R/r\overline R).\]
	Consider $rR\subset R\subset \overline R$ and $rR\subset r\overline R\subset \overline R$. Then 
	\begin{align*}
	l_R(\overline R/R) + l_R(R/rR) = l_R(\overline R/rR) 
	&= l_R(\overline R/r\overline R) + l_R(r\overline R/rR) \\
	&= l_R(\overline R/r\overline R) + l_R(\overline R/R).
	\end{align*}
	Note that $l_R(r\overline R/rR)=l_R(\overline R/R)$ as $r$ is a nonzerodivisor on $R.$ This shows that $l_R(R/rR)=l_R(\overline R/r\overline R)=l_{\overline R}(\overline R/r\overline R)$.
\end{proof}

Let $X,C,D$ be as above. Then $C \cap D$ is a finite set of points $p_1,\ldots,p_m$. We define 
\[C.D:=\Sigma_j \ i(C,D;p_j).\]

Now let $C:=\Sigma_ma_mC_m,D:=\Sigma_nb_nD_n$ be formal sums of irreducible curves $C_m,D_n$ on $X$ such that no $C_m$ is equal to any $D_n$ and $a_m,b_n$ are arbitrary non-zero integers. Such a formal sum is called a {\it divisor} on $X$. We define 
\[C.D:=\Sigma_{m,n} \ a_m.b_nC_m.D_n.\]

Let $X$ be a smooth projective surface and $C$ a reduced, irreducible curve on $X$. Our aim is to define $C.C$. We can find a non-zero rational function $f$ on $X$ such that $(f)+C$ is a divisor $D$ on $X$ such that the supports of $C,D$ have no common curve. We define $C.C:=C.D$. Using the well-known result that the orders of zeros and poles of a non-zero rational function on a smooth projective curve are equal we can show that $C.D$ is well-defined.\\

Let $X$ be a smooth projective surface and $C,D$ divisors on $X$. We say that $C,D$ are {\it linearly equivalent} (or {\it rationally equivalent}), written $C\sim D$, if the divisor $C-D=(r)$ for some non-zero rational function $r$ on $X$. Now if $C_1\sim C_2$ then for any divisor $D$ on $X$, we have $C_1.D=C_2.D$.\\

If $\pi:Y\rightarrow X$ is a proper surjective morphism between smooth projective surfaces, then for any divisors $C,D$ on $X$ we have $\pi^*C.\pi^*D=degree~\pi.(C.D)$. Here, $\pi^*C,\pi^*D$ are the scheme-theoretic pull-backs of $C,D$ by $\pi$.\\

Let $\pi:Y\rightarrow X$ be a surjective morphism between normal projective surfaces. Let $E_1,\ldots,E_m$ be all the irreducible curves on $Y$ such that $\pi(E_j)$ is a point in $X$ for every $j$. We call $E_j$ the exceptional curves for the morphism $\pi$. It can be shown that for any divisor $C$ on $X$, we have $\pi^*C.E_j=0$ for each $j$.\\

Now we come to an important basic result in the intersection theory on surfaces due to Patrick Du Val.

\begin{Lemma}
	Let $X$ be a normal projective surface, $Y$ a smooth projective surface and $f:Y\rightarrow X$ be a surjective morphism. Let $E_1,\ldots,E_m$ be all the exceptional curves on $Y$ for $f$. Then the intersection form on $\cup E_j$ is negative definite.
\end{Lemma}

\begin{proof}
	For simplicity, we will assume that all the curves $E_j$ map to the same point $p\in X$. The general case is similar. Let $r$ be a regular function on $X$ in a neighborhood of $p$. Then on $Y$ we have $D:=(r)_Y=C+\Sigma m_jE_j$, where $C$ is the part of the divisor $(r)_Y$ which does not contain any $E_j$ in its support. Then $(r)_Y.E_j=0$ for every $j$. Hence
	\[D.E_j=C.E_j+E_j.(\Sigma m_lE_l)\]
	for every $j>0$. Since $r$ is regular at $p$, we see that $C.E_j\geq 0$ for every $j>0$. Also, it can be shown easily that $C.E_j>0$ for some $j>0$. Thus, $(\Sigma m_jE_j).E_l\leq 0$ for every $l$ and strict inequality holds for some $l$. We will show that this implies that the intersection form on $\cup E_j$ is negative definite.
	
	Consider the symmetric quadratic form on an $m$-dimensional real vector space, with basis $x_1,\ldots,x_m$ given by $\Sigma \ \alpha_{ij}  m_im_jx_i.x_j$, where $\alpha_{ij}=E_i.E_j$. It suffices to show that this form is negative definite. We have\\
	(1) $\alpha_{ij}\geq 0$ if $i\neq j$,\\
	(2) $\Sigma_i \alpha_{ij}=(\Sigma m_iC_i).m_jE_j\leq 0$ for every $j$, and\\
	(3) $\Sigma_i \alpha_{ij}<0$ for some $j$.\\
	From this, we get
	\[\Sigma_{ij} \ \alpha_{ij}x_ix_j = \Sigma_j(\Sigma_i \ \alpha_{ij})x_j^2 - \Sigma_{i<j} \ \alpha_{ij}(x_i-x_j)^2.\]
	This shows the negative semi-definiteness of the intersection form. If the R.H.S. is $0$ for some real values $x_1,\ldots,x_m$ then (3) shows that $x_j=0$ if $\Sigma_i\alpha_{ij}<0$. But then $x_i=x_j$ for all $i,j$ and hence the result follows.
\end{proof}

{\bf The inequalities.}
Let $(R,M)$ be an algebraic or complex analytic local domain of dimension $d$ with $R/M\cong k$. In the complex analytic case, $k=\mathbb{C}$. Let $I\subset R$ be an $M$-primary ideal, and let $e(I)$ be the multiplicity of $I$. Let $\pi:X\rightarrow \Spec R$ be a resolution of singularities such that $\pi^*I$ is an invertible sheaf of ideals on $X$. If $\chara k=0$, or $d=2$, or $d=3$ and $\chara k>5$, then such a resolution of singularities exists \cite{L},\cite{A} respectively. Then $\pi^*I$ defines an effective divisor $D=\Sigma n_iD_i$ on $X$, where $D_i$ are irreducible divisors on $X$. The next result is the geometric interpretation of multiplicity proved by C.P. Ramanujam \cite{CPR}.

\begin{Theorem} \label{g1}
	$e(I)=(-1)^{d-1}.D^n$, where $D^n=D.D \ldots D$ n times.
\end{Theorem}

\begin{proof}
	We can find a minimal reduction $(x_1,\ldots,x_d)$ of $I$. By definition, $(x_1,\ldots,x_d)I^l=I^{l+1}$ for all $l$ large and $e((x_1,\ldots,x_d))=e(I).$ It is a standard result that $k[[x_1,\ldots,x_d]]\subset \hat R$ is a Noether normalization of degree $e(I)$, where $\hat R$ is the completion of $R$ with respect to $M$. This gives a finite morphism  
	\[\Spec \hat R \rightarrow S:=\Spec (k[[x_1,\ldots,x_d]])\] 
	of degree $e(I).$\\

	Let $\sigma:\tilde S\rightarrow S$ be the monoidal transform with center $(x_1,\ldots,x_d)$. Then the scheme-theoretic inverse image of the closed point of $\Spec S$ is a reduced divisor $E \cong {\bf P}^{d-1}$, and $E.E \ldots E=(-1)^{d-1}$.\\

	We can assume (by further blowing ups on $X$, if necessary) that there is a proper morphism of degree $e(I)$, say $\tilde\pi:X\rightarrow\tilde S$. By the property of minimal reduction, $(x_1,\ldots,x_d)$ and $I$ generate the same ideal sheaf on $X$. Thus, $\tilde\pi E=D$. By projection formula, $(\tilde\pi^* E)^d=(-1)^{d-1}e(I)$, i.e., $D^n=(-1)^{d-1}e(I)$.
\end{proof}

\begin{Remark} \label{divisor}{\rm 
	{\rm(1)} Let $(R,M)$ be the local ring of a $2$-dimensional rational singular point. Let $\pi:X\rightarrow \Spec R$ be a resolution of singularities. It is known that $\pi^*M$ is an invertible sheaf of ideals. M. Artin proved that the divisor of zeros of this ideal sheaf is the fundamental cycle $Z$ for $\pi$. By Theorem \ref{g1}, $Z^2=-e(R).$ This is one of the results proved by Artin about rational surface singular points. Hence C.P. Ramanujam's geometric interpretation of multiplicity is a vast generalization of Artin's result about multiplicity of a rational surface singularity.\\
	{\rm(2)} Let $(R,M)$ be a normal algebraic local domain and $I$ be an $M$-primary ideal in $R.$ Let $f: X \to \Spec R$ be a resolution of singularities such that $f^*I$,
	$f^*{\overline{I}}$ are locally principal, where $\overline{I}$ is the integral
	closure of $I$ in $R.$ Then $f^*I=f^*{\overline{I}}.$ Therefore, if $I\subset J$ are $M$-primary ideals such that $f^*I$, $f^*\overline{I}$, $f^*J$ and $f^*\overline{J}$ are locally principal and $f^*I, f^*J$ define the divisors $D$ and $E$ in $X$. Clearly, $\overline I\subset \overline J$. It will be proved in the course of the proof of Theorem 9.4 that $D=E$ implies that $\overline{I} = \overline{J}.$ \\ 
	{\rm(3)} Assume that $d=2$. Then the intersection form of $\cup D_i$ is negative definite. In particular, the irreducible components $D_1,D_2,\ldots$ are rationally independent. If $k=\mathbb{C}$, then the Chern classes $[D_1],[D_2],\ldots$ are even $\mathbb{Q}$-independent in the rational homology of $X$. This will be used in the proof of the next result.}
\end{Remark}

\begin{Theorem} \label{g2}
	Let $(R,M)$ be as above and $d=2$. Let $I,J$ be $M$-primary ideals. Then we have\\
	{\rm(1)} $e(IJ)^{1/2}\leq e(I)^{1/2} + e(J)^{1/2}$.\\
	{\rm(2)} Equality holds if and only if the integral closures $\overline{I^r}=\overline{J^s}$ for some $r,s\geq 1$.
\end{Theorem}

\begin{proof}
	Let $\pi:X\rightarrow \Spec R$ be a resolution of singularities such that $\pi^*I,\pi^*J$ are locally invertible sheaves of ideals. By Theorem \ref{g1}, if $D,E$ are the divisors on $X$ defined by $\pi^*I,\pi^*J$ then \[D^2=-e(I),E^2=-e(J),(D+E)^2=-e(I.J).\] 
	We want to show that 
	\[(-(D+E)^2)^{1/2} \leq (-D^2)^{1/2}+(-E^2)^{1/2}.\] 
	By squaring, this is equivalent to
	\[-(D+E)^2\leq -D^2-E^2+2(-D^2)^{1/2}(-E^2)^{1/2},\]
	i.e., $D.E\leq (-D^2)^{1/2}(-E^2)^{1/2}$, or $(D.E)^2\leq (-D^2).(-E^2)=D^2.E^2$. Since the intersection matrix of $D,E$ is negative definite, the result follows.\\ 
	Now suppose that the equality holds. Then $D,E$ are rationally dependent divisors. Hence there are positive integers $r,s$ such that $rD=sE$. This again uses the negative definiteness of the intersection form. Since the integral closures $\overline{I^r},\overline{J^s}$ are the unique largest ideals which define $rD,sE$ respectively, we get $\overline{I^r}=\overline{J^s}$.
\end{proof}

\begin{Theorem}
	Let $(R,M)$ be as in Theorem {\rm\ref{g2}} with $d=2$. Let $I,J$ be $M$-primary ideals. Then we have\\	
	{\rm(1)} $e(IJ)\leq 2e(I)+2e(J)$.\\
	{\rm(2)} Equality holds if and only if  $\overline I=\overline J$.
\end{Theorem}

\begin{proof}
	We want to prove \[-(D+E)^2\leq 2(-D^2)+2(-E^2).\]
	This is equivalent to $D^2+E^2-2D.E\leq 0$, i.e. $(D-E)^2\leq 0$. By the negative definiteness of the intersection form this follows. If equality holds, then $D=E$, i.e. $\overline I=\overline J$. In this case, $e(I^2)=4e(I)$. Since $2D=\pi^* (I^2)$, this follows from Theorem \ref{g1}.
\end{proof}

\section{ Bhattacharya function of two ideals}

Let $(R,\mathfrak{m})$ be a $d$-dimensional Noetherian local ring and $I,J$ be $\mathfrak{m}$-primary ideals. Recall that $H(r,s) = \ell(R/I^rJ^s)$ is called the Bhattacharya function of $I$ and $J.$ We prove

\begin{Theorem} \label{hilbertfunc}
	The function $H(r,s)$ is given by a polynomial of degree $d$ for all large $r,s\in \mathbb{N}.$ If we write this polynomial as
	\[P(r,s)=\frac{1}{d!}\left\{ e_0(I|J)r^d + \cdots + \binom{d}{i}e_i(I|J)r^{d-i}s^i + \cdots+ e_d(I|J)s^d\right\}  + \cdots \]
	then $e_i(I|J)$, for all $i=0,\ldots,d,$ are positive integers.
\end{Theorem}

In order to prove the theorem, we follow the arguments as given in \cite[Proposition 2.1]{teissier1973} by Risler-Teissier. For this, we need to first define superficial element.

\begin{Definition}
	Let $R$ be a Noetherian ring and $I,J$ be ideals in $R.$ An element $x \in I$ is superficial for $I,J$ if there exists $c>0$ such that for $r \geq c$ and $s \geq 0,$
	\[ (I^{r}J^{s} : x) \cap I^cJ^{s} = I^{r-1}J^{s}.\]
\end{Definition}

\begin{Theorem} \label{superficial}
	{\rm(Existence of superficial elements)} If $(R,\mathfrak{m})$ is a Noetherian local ring with infinite residue field, then superficial element for $I,J$ exist. Explicitly, there
	exists a non-empty Zariski-open subset $U$ of $I/\mathfrak{m} I$ and $c \in \mathbb{N}$ such that for any $x \in I$ with image in $U$, for all $r \geq c$ and all $s \geq 0$,
	\[ (I^{r}J^{s} : x) \cap I^cJ^{s} = I^{r-1}J^{s}.\]
	If $I$ is not contained in the prime ideals $\mathfrak{p}_1,\ldots,\mathfrak{p}_t$ of $R$, $x$ can be chosen to avoid the same prime ideals.
\end{Theorem}

For geometric interpretation of superficial element in the case of one ideal, we refer the reader to a paper by Romain Bondil \cite{bondil}.

One often studies Hilbert polynomials in low dimension and then uses induction on $\dim R.$ Superficial elements allow us to pass to lower dimensions. This is possible due to the next result.

\begin{Theorem} \label{sup}
	Let $x \in I$ be superficial for the pair $(I,J)$ with respect to the set of minimal primes of $R.$ Then for all large $r$ and $s,$ \vspace*{2mm} \\
	\vspace*{2mm}
	{\rm (1)} $\displaystyle \frac{I^rJ^s : x}{I^{r-1}J^s} \simeq (0: x)$\\ 
	{\rm (2)} $\displaystyle \ell\left(\frac{R}{I^rJ^s}\right) - \ell\left(\frac{R}{I^{r-1}J^s}\right) = \ell\left(\frac{R}{I^rJ^s+(x)}\right) - \ell((0 : x)).$
\end{Theorem} 

\begin{proof} 
	By Artin-Rees lemma, there exists $(m,n) \in \mathbb{N}^2$ such that for all $(r,s) \geq (m,n),$
	\[I^rJ^s \cap (x) \subseteq xI^{r-m}J^{s-n}.\]
	Let $a \in R$ such that  $ax \in I^rJ^s \cap (x).$ Then $ax \in  xI^{r-m}J^{s-n}.$ Write $ax=xp$ for some $p\in I^{r-m}J^{s-n}.$ Then $x(a-p)=0$ and hence $a \in I^{r-m}J^{s-n}+(0:x).$ Thus
	\[I^rJ^s : x \subseteq I^{r-m}J^{s-n}+(0:x).\]
	Since $x \in I$ is superficial with respect to $(I,J)$ there is a $c> 0$ such that for all $r \ge c$ and all $s \geq 0,$ $(I^rJ^s : x) \cap I^cJ^s = I^{r-1}J^s.$ As $I$ is $\mathfrak{m}$-primary, there exists $k>0$ such that $I^k \subseteq I^cJ^n$ and hence $I^k I^{r-m}J^{s-n} \subseteq I^cJ^s$ for all $(r,s) \ge (m,n)$. It follows that for all $r \ge m+k$, $s \ge n$,
	\[ (I^rJ^s : x) \subseteq (0 : x) + I^c J^s .\] 
	Hence for all $r \ge m+k$, $s \ge n,$ 
	\[(I^rJ^s : x) = (I^rJ^s : x) \cap ((0 : x) + I^cJ^s) = (0 : x) + I^{r-1}J^s.\] Therefore for large $r$ and $s$,
	\begin{align*}
	\frac{I^rJ^s : (x) } {  I^{r-1}J^s } 
	= \frac{(0 : x) + I^{r-1}J^s}{I^{r-1}J^s} 
	\simeq \frac{(0 : x)}{(0 : x) \cap I^{r-1}J^s}.
	\end{align*}
	Consider
	\begin{align*}
	(0 : x) \cap I^{r-1}J^s 
	\subseteq \left( \bigcap_{r \ge 1} (I^rJ^s : x) \right) \cap I^{r-1}J^s
	&\subseteq \left( \bigcap_{r \ge 1} (I^rJ^s : x) \right) \cap I^c J^s \\
	&= \bigcap_{r \ge 1} I^{r-1}J^s \\ 
	&\subseteq \bigcap_{r \ge 1} I^{r-1} = (0).
	\end{align*}
	Hence it follows that for large $r$ and $s$, $(I^rJ^s : (x))/I^{r-1}J^s \simeq (0:x)$. By avoiding the minimal associated primes of $R,$ we ensure $\dim R/(x)=\dim R-1.$ For the second assertion, use the exact sequence of $R$-modules,
	\[0\longrightarrow  \frac{ I^rJ^s : (x) } { I^{r-1}J^s }\longrightarrow \frac{R}{   I^{r-1}J^s }\stackrel {\mu_x} \longrightarrow \frac{R} { I^rJ^s } \longrightarrow \frac{R}{ I^rJ^s+(x) }\longrightarrow 0.\]
\end{proof}

{\bf Proof of Theorem \ref{hilbertfunc}:} Proceed by induction on $\dim R =d.$ If $d=0$, then for large $r$ and $s$, $I^rJ^s=0.$ Hence $H(r,s) = \ell(R) = P(r,s)$ for large $r,s.$ Thus, $P(r,s)$ is a degree zero rational polynomial. Let $d>1$ and $\{\mathfrak{p}_1,\ldots,\mathfrak{p}_t \}$ be the set of minimal primes of $R.$ As $I$ is $\mathfrak{m}$-primary, it follows that $I \not\subseteq \cup_{i=1}^t \mathfrak{p}_i.$ Using Theorem \ref{superficial}, there exists $x \in I \setminus \cup_{i=1}^t \mathfrak{p}_i$ and $c>0$ such that for all $r \ge c$ and $s \ge 0,$
\[ (I^rJ^s : x) \cap I^{c}J^s = I^{r-1}J^s. \]
From Theorem \ref{sup}, 
\[H(R/(x), (r,s)) = H(R,(r,s))- H(R,(r-1,s)) +\ell(0: x).\] 
By the choice of $x$, $\dim R/(x) = d-1$ and by induction hypothesis, $H(R/(x),(r,s))$ is a rational polynomial $Q(R/(x),(r,s))$ of degree $d-1$ for large $r,s$ such that every monomial of total degree $d-1$ in the polynomial has a positive coefficient. Let $H(R/(x),(r,s)) = Q(R/(x),(r,s))$ for all $r \ge r_0$ , $s \ge s_0.$ One can now conclude that for $r \ge r_0, s \ge s_0$,
\[ H(R,(r,s)) = H(R,(r_0,s)) + \sum_{i=r_0+1}^{r} \left( Q \left(\frac{R}{(x)},(i,s) \right) - \ell(0:x) \right). \]
Since $H(R,(r_0,s))$ and $\sum_{i=r_0+1}^{r} Q(R/(x),(i,s))$ are rational polynomials of degree $d$ (see \cite[Lemma 11.1.2]{swansonHuneke}) with every monomial of total degree $d$ having a positive coefficient, we are done.

\begin{Remark}
	Let $e'_0,\ldots,e'_d$ denote the mixed multiplicities of $I$ and $J$ as ideals in $R/(x).$ Using part (2) of Theorem \ref{sup}, it follows that, $e_0=e'_0, \ldots, e_{d-1}=e'_{d-1}.$ 
\end{Remark}

Given ideals $I,J$ of a ring $R,$ Rees introduced the notion of a Rees superficial element for the pair $(I,J)$ in \cite{rees1984}. 

\begin{Definition}
	Let $R$ be a Noetherian ring and $I,J$ be ideals in $R.$ An element $x \in I$ is Rees superficial for $I,J$ if there exists $c>0$ such that for $r \geq c$ and $s \geq 0,$
	\[ I^{r}J^{s} \cap (x)R = xI^{r-1}J^{s}.\]
\end{Definition}

\begin{Lemma}[{\rm{\cite[Lemma 1.2]{rees1984}}}]\label{Rees' lemma} 
	{\rm(Existence of Rees superficial element)} Let $(R,\mathfrak{m})$ be a local ring with $R/\mathfrak{m}$ infinite Let  $I,J$ be ideals of $R$ and let $\mathfrak{p}_1,\ldots, \mathfrak{p}_t$ be prime ideals which  do not contain $IJ.$ Then there exist $x\in I\setminus (\mathfrak{p}_1 \cup \cdots \cup \mathfrak{p}_t \cup \mathfrak{m} I)$ and $c>0$ such that for all $r \geq c$ and $s \ge 0$,
	\[ I^rJ^s \cap (x)R = xI^{r-1}J^s. \]
\end{Lemma}

\begin{Theorem} \label{reesSup}  
	Let $I,J$ be $\mathfrak{m}$-primary ideals of a ring $(R,\mathfrak{m})$ and $x \in I$ be a Rees superficial element for the pair $(I,J).$ If $\Ann(x) \subseteq \Ann(I)$, then for all large $r$ and $s,$
	\begin{align*}
	\ell \left(\frac{R}{I^rJ^s}\right) - \ell \left(\frac{R}{I^{r-1}J^s}\right) = \ell \left(\frac{R}{I^rJ^s+(x)}\right) - \ell((0:x)).
	\end{align*}
\end{Theorem}

\begin{proof} 
	As $x$ is Rees superficial for $(I,J)$, there exists $c>0$ such that for all $r\geq c$ and $s\geq 0,$ $I^rJ^s \cap (x) = x I^{r-1}J^s.$ Using the arguments as in proof of Theorem \ref{sup}, it is sufficient to show that for large $r,s,$ 
	\[\frac{I^rJ^s \colon (x)}{I^{r-1}J^s} \simeq (0:x).\]
	Indeed, let  $b \in I^rJ^s \colon (x).$ Then for all $r \ge c$ and $s \ge 0$, $bx \in I^rJ^s \cap (x) = xI^{r-1}J^s.$ Write $bx = xu$, for some $u \in I^{r-1}J^s.$ This implies that $(b-u) \in (0 : x)$ and hence for all $r \ge c$ and $s \ge 0,$ $b \in (0:x)+ I^{r-1}J^s.$ Therefore for all $r \ge c$ and $s \ge 0$,
	\begin{align*}
	\frac{I^rJ^s \colon (x)}{I^{r-1}J^s} 
	= \frac{(0 : x) + I^{r-1}J^s}{I^{r-1}J^s}
	\simeq \frac{(0 : x)}{(0 : x) \cap I^{r-1}J^s}.
	\end{align*}
	Using Artin-Rees lemma, there exist $r_0,s_0$ such that  for all $r\geq r_0$ and $s\geq s_0,$
	\begin{align*}
	(0 : x) \cap I^{r-1}J^s	\subseteq (0 : x)I^{r-1-r_0}J^{s-s_0}=0.
	\end{align*}
	The last equality follows as $(0:x) \subseteq (0:I).$ Hence it follows that for large $r$ and $s$, $(I^rJ^s : (x))/I^{r-1}J^s \simeq (0:x)$.
\end{proof}

\section{Minkowski's equality for multiplicity of ideals in one-dimensional local rings}

In this section, we shall prove that in a one-dimensional local ring $(R,\mathfrak{m}),$  $e(IJ)=e(I)+e(J)$ for all $\mathfrak{m}$-primary ideals $I, J\subset R.$  Recall that the zeroeth local cohomology module of $R$ with respect to $\mathfrak{m}$ is the ideal \[H_\mathfrak{m}^0(R)=\{x\in R\mid x\mathfrak{m}^n=0 \text{ , for some } n\in \mathbb{N}\}.\]
\begin{Lemma} \label{H0}
	Let $R$ be a Noetherian local ring of dimension $d\geq 1$ and $I$ be an $\mathfrak{m}$-primary ideal. Set $S = H^0_{\mathfrak{m}}(R).$ Then 
	\[e(I, R) = e((I+S)/S, R/S).\]
\end{Lemma}

\begin{proof}
	Observe that 
	\begin{align*}
	\ell \left(\frac{R/S}{(I^n+S)/S}\right) 
	= \ell\left(\frac{R}{I^n+S}\right) 
	&= \ell\left(\frac{R}{I^n}\right) - \ell\left(\frac{I^n+S}{I^n}\right) \\
	&= \ell\left(\frac{R}{I^n}\right) - \ell\left(\frac{S}{I^n \cap S}\right).
	\end{align*}
	Using Artin-Rees lemma, there exists $n_0 \in \mathbb{N}$ such that for all $n$ large, 
	 \[I^n \cap S = I^{n-n_0}(I^{n_0} \cap S) \subseteq I^{n-n_0}S=0.  \]
	Therefore for $n$ large, $\ell(R/(I^n+S)) = \ell(R/I^n) - \ell(S).$ Divide by $n^d/d!$ and take limit to get the required result.
\end{proof}

\begin{Proposition}
	Let $(R,\mathfrak{m})$ be a 1-dimensional local ring and $I,J$ be $\mathfrak{m}$-primary ideals. Then
	\[e(IJ) = e(I)+e(J).\]
\end{Proposition}
	
\begin{proof}
	Using Lemma \ref{H0}, we may pass to $R/H^0_{\mathfrak{m}}(R)$ and so assume that $R$ is Cohen-Macaulay (this is true because $\mathfrak{m}$ is not a associated prime of $H^0_{\mathfrak{m}}(R)$). Without loss of generality, we may assume that the residue field $k=R/\mathfrak{m}$ is infinite. Let $x\in I$ and $y\in J$ so that $(x)$ is a minimal reduction of $I$ and $(y)$ is a minimal reduction of $J.$ Then $(xy)$ is a minimal reduction of $IJ.$ Hence 
	\[e(IJ)=e(xy)=\lim_{n \rightarrow \infty} \ell(R/x^ny^n)/n.\] 
	Consider the exact sequences
	\[0\longrightarrow (x^n)/(x^ny^n)\longrightarrow R/(x^ny^n)\longrightarrow R/(x^n)\longrightarrow 0.\]
	Since $x^n$ is a nonzerodivisor, $(x^n)/(x^ny^n) \simeq R/(y^n)$ and so it follows that
	\[\ell(R/(x^ny^n)) = \ell(R/(x^n))+\ell(R/(y^n)).\] 
	Divide by $n$ and take limit to see that $e(IJ) = e(I)+e(J).$  
\end{proof}

\section{ Teissier's approach to Minkowski's  inequalities} 
\textbf{Mixed multiplicities of ideals:}  Let $I$ and $J$ be $\mathfrak{m}$-primary ideals of a $d$-dimensional local ring $(R,\mathfrak{m}).$  We know that the function $H(r,s)=\ell(R/I^rJ^s),$ for all large $r,s$,  is given by a polynomial 
\[P(r,s)=\frac{1}{d!}\left\{ e_0(I|J)r^d + \cdots + \binom{d}{i}e_i(I|J)r^{d-i}s^i+\cdots+ e_d(I|J)s^d\right\}  + \cdots. \]
Here, $e_i(I|J)$ for $i=0,1,\ldots, d$ are positive integers called the {\em mixed multiplicities of $I$ and $J.$} We prove some basic properties of mixed multiplicities in the next result.

\begin{Lemma} 
	We have 
	{\rm (1) } $e_0(I|J)=e(I), e_d(I|J)=e(J).$ \\
	{\rm (2) } For all integers $r, s\geq 0,$ 
	\[e(I^rJ^s)=e(I)r^d+\cdots+e_i(I|J)\binom{d}{i}r^{d-i}s^i+ \cdots+ e(J)s^d.\]
	{\rm (3)} For all positive integers $p,q,$ $e_i(I^p|J^q)=e_i(I|J)p^{d-i}q^i$ for all $i=0,1,\ldots,d.$\\
	{\rm (4)} For all $i=0,1,\ldots, d$, we have $e_i(I|I)=e(I).$\\
	{\rm (5)}  If $I$ is a reduction of $K$ and $J$ is a reduction of $L$, then $e_i(I|J)=e_i(K|L)$ for all $i=0,1,\ldots,d.$
\end{Lemma}

\begin{proof}
	(1) Let $H(r,s)=P(r,s)$ for all $r\geq l$ and $s\geq k.$ Then for all $r\geq l,$   
	\[\ell(R/I^rJ^k)=e_0(I|J)\frac{r^d}{d!}+\text{ terms of degree }  < d.\]
	Let $m\in \mathbb{N}$ such that $I^m\subseteq J^k.$ Then $I^{r+m}\subseteq I^rJ^k\subseteq I^r$ for all $r\geq 0.$ Hence for $r\geq l,$ we have
	\[\ell(R/I^{r+m}) \geq H(r,k)\geq\ell(R/I^r).\]
	Divide by $r^d/d!$ and take limit $r\to \infty$ to see that $e(I)\geq e_0(I|J)\geq e(I).$ Hence $e_0(I|J)=e(I).$ By symmetry, $e(J)=e_d(I|J).$ \\
	(2) We may assume that $r,s\geq 1.$ For large $n,$
	\begin{align*}
	\ell(R/I^{rn}J^{sn})
	&= \frac{1}{d!}\sum_{i=0}^de_i(I|J)\binom{d}{i}(rn)^{d-i}(sn)^i+\cdots \\
	&=\frac{n^d}{d!}\sum_{i=0}^de_i(I|J)\binom{d}{i}r^{d-i}s^i+\cdots
	\end{align*}
	It follows that for all $r,s\geq 1,$ $e(I^rJ^s)=\sum_{i=0}^de_i(I|J)\binom{d}{i}r^{d-i}s^i.$\\
	(3) Using the formula for $e(I^rJ^s)$ we have
	\begin{align*}
	\sum_{i=0}^d \binom{d}{i}e_i(I|J)(rp)^{d-i}(sq)^i
	=e(I^{pr}J^{qs})
	=\sum_{i=0}^d\binom{d}{i}e_i(I^p|J^q)r^{d-i}s^i.
	\end{align*}
	Equate the coefficients of $r^{d-i}s^i$ to get $e_i(I^p|J^q)=p^{d-i}q^ie_i(I|J)$ for all $i=0,\ldots,d.$\\
	(4) Put $I=J$ in (2) to get
	\[ \sum_{i=0}^de_i(I|I)\binom{d}{i}r^{d-i}s^i = e(I^{r+s}) = e(I)(r+s)^d = \sum_{i=0}^d\binom{d}{i}e(I)r^{d-i}s^i. \]
	 Equate the coefficients of $r^{d-i}s^i$ to see that $e_i(I|I)=e(I)$ for all $i=0,1,\ldots,d.$\\
	(5) Since $I$ is a reduction of $K$ and $J$ is a reduction of $L,$ $I^rJ^s$ is a reduction of $K^rL^s$ for all $r,s\geq 1.$ Hence
	\[ \sum_{i=0}^d\binom{d}{i}e_i(I|J)r^{d-i}s^i=e(I^rJ^s)=e(K^rL^s) = \sum_{i=0}^d\binom{d}{i}e_i(K|L)r^{d-i}s^i. \]
	Equate the coefficients of $r^{d-i}s^i$ to see that $e_i(I|J)=e_i(K|L)$ for all $i=0,1,\ldots,d.$
\end{proof}

\noindent{\bf Teissier's approach to Minkowski's inequality:} 
Put $r=s=1$ in part (2) of the above lemma  to get
 \[e(IJ)=e(I)+\binom{d}{1}e_1(I|J)+\cdots+e_i(I|J)\binom{d}{i}+ \cdots+ e(J).\]
Teissier compared the expansion 
\[(e(I)^{1/d}+e(J)^{1/d})^d = e(I)+\cdots+ \binom{d}{i}e(I)^{(d-i)/d}e(J)^{i/d}+\cdots + e(J)\]
with the formula for $e(IJ)$ and proposed the following:

\begin{Conjecture}[\bf Teissier's first conjecture] \label{tc1}
	Let $(R,\mathfrak{m})$ be a $d$-dimensional Noetherian local ring and $I, J$ be $\mathfrak{m}$-primary ideals. Then for all $i=0,1,\ldots, d,$
	\[e_i(I|J)^d\leq e(I)^{d-i}e(J)^i.\]
\end{Conjecture}

It is clear that Teissier's first conjecture implies Minkowski's inequality. Teissier approached his first conjecture via his second.

\begin{Conjecture}[\bf Teissier's second conjecture] \label{tc2}
	Let $I$ and $J$ be $\mathfrak{m}$-primary ideals of a Noetherian local ring of dimension $d\geq 2.$ Put $e_i=e_i(I|J).$ Then
	\[ \frac{e_1}{e_0}\leq \frac{e_2}{e_1}\leq \frac{e_3}{e_2}\leq \cdots\leq \frac{e_{d}}{e_{d-1}}.\]
\end{Conjecture}

\begin{Proposition}
	Conjecture  \ref{tc2} implies the Conjecture  \ref{tc1}  .
\end{Proposition}

\begin{proof}
	Apply induction on $d.$ The two conjectures are equivalent for $d=2.$ Assume that the conjecture \ref{tc2} for $d-1$ and $d\geq 3$, implies the conjecture \ref{tc1}. Then for $i=0,1,\ldots, d-1,$ $e_i^{d-1}\leq e_0^{d-1-i}e_{d-1}^i.$ Note that 
	\[\frac{e_1}{e_0}\frac{e_2}{e_1}\cdots \frac{e_i}{e_{i-1}}=\frac{e_i}{e_0}\leq 
	\left(\frac{e_{d}}{e_{d-1}}\right)^i.\]
	Hence 
	\[e_i^d\leq e_0^{d-1-i}e_ie_{d-1}^i=e_0^{d-i}e_{d-1}^i\frac{e_i}{e_0}\leq e_0^{d-i}e_{d-1}^i \left(\frac{e_{d}}{e_{d-1}}\right)^i=e_0^{d-i}e_d^i.\]
\end{proof}

\section{ Rees-Sharp proof of Minkowski's inequalities}
In this section, we shall prove Minkowski's  inequality for multiplicities in Noetherian local rings. We shall use Lech's formula for the multiplicity of an ideal generated by a system of parameters. Let $a_1, a_2,\ldots,a_d$ be a system of parameters in a $d$-dimensional local ring and $I=(a_1,a_2,\ldots,a_d).$ Then the Lech's formula for $e(I)$ is:
\[e(I)=\lim_{n\to \infty}\frac{\ell(R/(a_1^n, a_2^n,\ldots, a_d^n))}{n^d}.\]
For $n\in \mathbb{N},$ we write $I^{[n]}=(a_1^n,a_2^n,\ldots,a_d^n).$

\begin{Lemma}\label{Lengthlemma}
	Let $(R,\mathfrak{m})$ be a two-dimensional Noetherian local ring, $I=(a,b)$ and $J$ be $\mathfrak{m}$-primary ideals. Then for all $n\geq 1,$
	\[\ell(R/(J^n:I^{[n]})) \leq \ell(R/I^{[n]})+2\ell(R/J^n)-\ell(R/I^nJ^n).\]
\end{Lemma}

\begin{proof}
	Use the following diagram
	\begin{align*}
	\xymatrix@=4mm{
	R \ar@{-}[dr]  \ar@{-}[d] \\
	I^nJ^n \ar@{-}[d] & I^{[n]} \ar@{-}[ld] \\
	I^{[n]}J^n
	}
	\end{align*}
	to see that
	\[\ell(R/I^nJ^n)+\ell(I^nJ^n/I^{[n]}J^n)=\ell(R/I^{[n]})+\ell(I^{[n]}/I^{[n]}J^n).\]
	In order to estimate $\ell(I^{[n]}/I^{[n]}J^n),$ consider the complex of $R$-modules

	\[0\longrightarrow R \stackrel{f}  \longrightarrow R/J^n\oplus R/J^n \stackrel{g} \longrightarrow \frac{(a^n,b^n)} {(a^nJ^n+b^nJ^n)}\longrightarrow 0\] 
	where $f(r)=(r[b^n], -r[a^n])$ and $g([r],[s])=[ra^n+sb^n]$ for all $r, s \in R.$ Then $\ker f=(J^n:I^{[n]})$ and hence the image of $f$ is isomorphic to $R/ (J^n:I^{[n]}).$ Therefore
	\begin{align*}
	\ell(I^{[n]}/I^{[n]}J^n) \leq 2 \ell(R/J^n)-\ell(R/J^n:I^{[n]})
	\end{align*}
	and hence $\ell(R/J^n:I^{[n]}) \leq  \ell(R/I^{[n]}) -\ell(R/I^{[n]}J^n) +2 \ell(R/J^n).$
\end{proof}

\begin{Proposition}\label{e(IJ)}
	Let $(R,\mathfrak{m})$ be a $2$-dimensional local ring and $I,J$ be $\mathfrak{m}$-primary ideals. Then 
	\[e(IJ)\leq 2 e(I)+2 e(J).\]
\end{Proposition}

\begin{proof}
	We may assume without loss of generality that $R/\mathfrak{m}$ is infinite. Let $(a,b)$ be a minimal reduction of $I.$ Then $(a,b)J$ is a reduction of $IJ.$ Hence $e(IJ)=e((a,b)J)$ and $e(I)=e(a,b).$ Therefore we may assume that $I=(a,b).$ By Lemma \ref{Lengthlemma}, we have
	\begin{align*}
	\ell(R/I^nJ^n)
	&\leq \ell(R/I^{[n]})+2\ell(R/J^n)-\ell(R/J^n:I^{[n]})\\
	&\leq \ell(R/I^{[n]})+2\ell(R/J^n).
	\end{align*}
	Divide by $n^2/2$ and take the limit $n\rightarrow \infty.$ This gives $e(IJ)\leq 2 e(I)+ 2 e(J).$ 
\end{proof}

\begin{Theorem}
	Let $(R,\mathfrak{m})$ be a $2$-dimensional local ring and $I,J$ be $\mathfrak{m}$-primary ideals. Then 
	\[e_1(I|J)^2\leq e(I) e(J).\]
\end{Theorem}

\begin{proof}
	Using Proposition \ref{e(IJ)}  for $I^r$ and $J^s$, we get
	\begin{align*}
	r^2e(I)+2e_1(I|J)rs+s^2e(J) 
	=  e(I^rJ^s)
	&\leq 2 e(I^r)+2 e(J^s)\\
	&= 2r^2e(I)+2s^2e(J).
	\end{align*}
	Therefore for all  $r,s\in \mathbb{N}$, we have
	\[f(r,s):= r^2e(I)-2rs e_1(I|J)+e(J) s^2\geq 0.\]
	Hence the discriminant of $f(r,s),$ namely, $4e_1(I|J)^2-4e(I)e(J)\leq 0.$ This implies that $e_1(I|J)^2 \leq e(I)e(J).$
\end{proof}

\begin{Theorem}[\bf Rees-Sharp] 
	Let $(R,\mathfrak{m})$ be a local ring of dimension $d\geq 2.$ Let $I$ and $ J$ be $\mathfrak{m}$-primary ideals. Then Teissier's second conjecture is true. 
\end{Theorem}

\begin{proof}
	Apply induction on $d.$ We have established the conjecture for  $d=2.$ Let $d\geq 3.$ Let $a\in I$ be superficial for $I, J$ and the set of minimal primes of $R.$ Then for $i=0,1,\ldots, d-1,$
	\[e_i(I/(a) \mid J+(a)/(a))=e_i(I|J).\]
	Hence 
	\[ \frac{e_1}{e_0}\leq \frac{e_2}{e_1}\leq \frac{e_3}{e_2}\leq \cdots\leq \frac{e_{d-1}}{e_{d-2}}.\]
	If we use a superficial element from $J, $ we obtain the remaining inequality by symmetry.
\end{proof}

\section{ Complete ideals and discrete valuation rings }

\begin{Definition} 
	A local domain $(S,\mathfrak{n})$  is said to dominate a local domain $(R,\mathfrak{m})$ birationally if $R \subset S \subset K$ where $K$ is the fraction field of $R$ and $\mathfrak{n} \cap R =\mathfrak{m}.$  If $S$ birationally dominates $R$, we write $ S \succ R$ or $R \prec S.$ 
\end{Definition}

\begin{Proposition}\label{dvrdom} 
	Let $(R,\mathfrak{m})$ be a local domain of positive dimension. Then there is a discrete valuation ring $(V,\mathfrak{n})$ birationally dominating $(R,\mathfrak{m}).$
\end{Proposition}

\begin{proof}
	First, we show that there exists an $x \in \mathfrak{m} $ such that $x^k \notin \mathfrak{m}^{k+1}$ for all $k \geq 1.$ Let $\mathfrak{m}=(x_1,x_2, \ldots,x_n)$ and assume by way of contradiction that $x_i^{k_i} \in \mathfrak{m}^{k_i+1}$ for some $k_i$ and for all $i=1,2, \ldots,n.$ Let $k=\max k_i.$ Since ${\bf x}^{[k]}:=(x_1^k,x_2^k, \ldots,x_n^k)$ is a reduction of $\mathfrak{m}^k$, there exists an $s$ such that  ${\bf x}^{[k]}\mathfrak{m}^{ks}= \mathfrak{m}^{ks+k}.$ Hence $\mathfrak{m}^{ks+k} \subset \mathfrak{m}^{sk+k+1}$ which yields $\mathfrak{m}^{ks+k}=0.$ This is a contradiction as $\dim R \geq 1.$ Thus we may assume without loss of generality that for $x_1=x$, $x^k \notin \mathfrak{m}^{k+1}$ for all $k.$

	The ring $S=R[\mathfrak{m}/x]=R[x_2/x,x_3/x, \ldots ,x_n/x]$ is called a {\em monoidal transform} of $R.$ It is easy to see that $S = \{ b/x^{k}: b \in \mathfrak{m}^k\, \mbox{ for some } k\}.$ The ideal $xS=\mathfrak{m} S$ is a proper ideal. Indeed, if  $1 \in xS$ then $1=bx/x^d$ for some $d \geq 1$ and  $b \in \mathfrak{m}^{d}.$ Hence $x^d \in \mathfrak{m}^{d+1}$, contradicting the choice of $x.$ Thus $xS$ is a height one ideal of $S.$ Let $Q$ be a minimal prime of $xS.$ By Krull-Akizuki theorem, the integral closure $T$ of $S_Q$ in its fraction field $K$ is a one dimensional Noetherian domain. Let $N$ be a maximal ideal of $T$ contracting to the maximal ideal  of $S_Q.$ Then $NT_N \cap R =\mathfrak{m}$ and hence $T_N$ is the desired discrete valuation ring birationally dominating $R.$ 
\end{proof}

\begin{Theorem} [{\bf Lipman's theorem}, \cite{lipman}]\label{Lipman's theorem} 
	Let $S$ be a Noetherian domain with fraction field K and let $I$ be a proper ideal of $S.$ Then
	\[\overline{I}=\bigcap_V IV\cap S \]
	where the intersection is  over all discrete valuation rings V in K such that  $V \supset S.$
\end{Theorem}

\begin{proof} 
	Since principal ideals in integrally closed domains are complete and intersections of complete ideals are complete, the ideal $J$ on the right hand side of the above equation is complete. Hence $\overline{I} \subseteq J.$ Conversely let $x \notin \overline{I}.$ Then we find a discrete valuation ring $V \supset S $ in $K$ such that $x \notin IV.$ Put $T=S[Ix^{-1}].$ Then $x^{-1}IT$ is a proper ideal of $T.$ Indeed, if  $x^{-1}IT=T,$ then $1=a_1/x + a_2/x^2 + \cdots + a_n/x^n,$ where $a_i \in I^i$ for $i=1,2, \ldots, n.$ Hence $x^n=a_1x^{n-1}+a_2x^{n-2}+ \cdots + a_n$ which shows that $x \in \overline{I}.$ This is a contradiction. Pick a minimal prime $Q$ of $x^{-1}IT.$ By Proposition \ref{dvrdom}, there exists a discrete valuation ring  $(V,\mathfrak{n})$  such that  $V \succ T_Q.$ Hence $ x^{-1}IT \subset Q \subset QT_Q =\mathfrak{n} \cap T_Q$ and $x^{-1}IV \subseteq \mathfrak{n}.$ Thus $x \notin IV.$
\end{proof}

\section{Minkowski equality in dimension two}
If $I$ and $J$ are $\mathfrak{m}$-primary ideals of a local ring $(R,\mathfrak{m})$ and $J$ is a reduction of $I$ then $e(I)=e(J).$ Moreover, a deep theorem of Rees asserts that in case $R$ is quasi-unmixed and $J\subset I$ with $e(I)=e(J)$, then 
$\overline{I}=\overline{J}.$ It is natural to ask for a numerical criterion for $\overline{I}=\overline{J}$ if there is no containment relation among $I$ and $J.$ We shall see that if $R$ is quasi-unmixed and $e(I)=e_i(I|J)$ for all $i=1,\ldots,d$, then $\overline{I}=\overline{J}.$ We shall prove this in dimension $2$ in this section.

\begin{Proposition}\label{equal}
	Let $R$ be a $d$-dimensional Noetherian local ring and $I$, $J$ be $\mathfrak{m}$-primary ideals of $R$. Suppose there exist positive integers $r$, $s$ such that $\overline{I^r} = \overline{J^s}$. Then 
	\[\frac{e_1}{e_0}=\frac{e_2}{e_1}=\cdots=\frac{e_d}{e_{d-1}}=\frac{r}{s}.\]
\end{Proposition}

\begin{proof} 
	Let $\overline{I^r}=\overline{J^s}.$ Then for all $i=0,1,\ldots,d,$
	\[r^{d-i}s^ie_i = e_i(\overline{I^r}|\overline{J^s}) = e_i(\overline{I^r}|\overline{I^r}) = r^de(I).\]
	This implies that $e_i/e_{i-1}=r/s$ for all $i=1,\ldots,d.$
\end{proof}

Rees and Sharp proved that the converse of the above theorem is true for two-dimensional quasi-unmixed local rings. They used the theory of  degree functions and $\mathfrak{m}$-valuations. Let $(R,\mathfrak{m})$ be a Noetherian local ring of dimension $d.$  Let $k(\mathfrak{p})$ denote the field of fractions of  $R/\mathfrak{p}$ where $\mathfrak{p}$ is a prime ideal of $R$ and let $\mathbb{Z}$ denote the additive group of integers.

\begin{Definition} 
	An $\mathfrak{m}$-valuation of a Noetherian local ring $(R,\mathfrak{m})$ is a map $v: R\to \mathbb{Z}\cup \{\infty\}$ which is a composition of the natural ring homomorphism $R\to R/\mathfrak{p}$, for a minimal prime $\mathfrak{p}$ of dimension $d$ and a valuation $v: k(\mathfrak{p})\to \mathbb{Z}\cup \{\infty\}$ such that {\rm (1)} $v(x)\geq 0$ for all $x\in R/\mathfrak{p}$, {\rm  (2)}  $v(x)> 0$ for all $ x\in \mathfrak{m}/\mathfrak{p}$ and, {\rm (3) } the residue field $K_v$ of $v$ is a finitely generated extension of $R/\mathfrak{m}$ of transcendence degree $d-1.$
\end{Definition}

\begin{Corollary} [{\cite[Proposition 1.1]{lipman}}] \label{val}
	Let $R$ be a complete, universally catenary, Noetherian local domain of dimension $d.$ Let $I$ be a proper ideal of $R.$ Then for every $x \notin \overline{I},$ there exists an $\mathfrak{m}$-valuation $v$ of $R$ such that $v(x) < v(I).$ 
\end{Corollary}

\begin{proof}
	Let $x \notin \overline{I}.$ As in the proof of Proposition \ref{dvrdom} and Theorem \ref{Lipman's theorem}, there exists a discrete valuation domain $T_N$ birationally dominating $R$ such that $x \notin IT_N$ and $NT_N$ contracts to the maximal ideal $\mathfrak{m}$ of $R.$ It is sufficient to show that $\trdeg _{k(m)} k(NT_N)=d-1.$ Since $R$ is a complete local domain, it is excellent and hence $T$ is finitely generated over $R.$ As $R$ is universally catenary, using dimension formula,
	\[ \trdeg_{k(m)} k(NT_N) = \hei \mathfrak{m} + \trdeg_{R}T - \hei N = d - 1. \]
\end{proof}

\begin{Lemma}\label{1}
	Let $(R,\mathfrak{m})$ be a $2$-dimensional quasi-unmixed local ring. Let $v$ be an $\mathfrak{m}$-valuation of $R.$ For a positive integer $i,$ let $c(v)_i = \{ a \in R \mid v(a) \ge i \}.$ Then there exists a positive real number $\lambda $ such that for all positive integers $i,$
	\begin{align*}
	\ell\left(\frac{R}{c(v)_i}\right) \ge \lambda i^2.
	\end{align*}
\end{Lemma}

\begin{proof}
	Let $(V,\mathfrak{n})$ be the discrete valuation ring corresponding to the valuation $v.$ Let $\mathfrak{p}$ be the minimal prime ideal of $R$ such that $R/\mathfrak{p} \subseteq V.$ Write $(\_)'$ to denote images in $R/\mathfrak{p}.$ There exist $x_1,x_2 \in R \setminus \mathfrak{p}$ such that $x'_1/x'_2$ in $V/\mathfrak{n}$ is transcendental over $R/\mathfrak{m}.$ This implies that $v(x_1) = v(x_2) = n$, say, for some positive integer $n.$ Observe that $n>0.$ If not, then $v(x_1)=v(x_2)=0$ implies that $x'_1,x'_2 \in V \setminus \mathfrak{n}.$ Therefore $x'_1/x'_2 \in R/\mathfrak{m}$ giving a contradiction. 
	
	Let $r$ be a positive integer. By definition, $\mathfrak{m} c(v)_{nr} \subseteq c(v)_{nr+1}$ implying that $c(v)_{nr}/c(v)_{nr+1}$ is an $R/\mathfrak{m}$-vector space. We claim that $x_1^r, x_1^{r-1}x_2, \ldots, x_2^r$  are linearly independent over $R/\mathfrak{m}.$ If not, then there exist $a_1, \ldots, a_{r+1} \in R/\mathfrak{m}$, not all zero, such that 
	\[a_1 x_1^r + a_2 x_1^{r-1}x_2 + \cdots + a_{r+1} x_2^r =0\]
	in the vector space. Dividing by $x_2^r$, we get $x_1/x_2$ is algebraic over $R/\mathfrak{m}$, a contradiction. Since $v(x_1^{r-i}x_2^i) = nr$ for all $i=0,1,\ldots,r$, it follows that 
	\[\dim_{R/\mathfrak{m}} (c(v)_{nr}/c(v)_{nr+1}) \ge r+1.\] 
	Let $i$ be a positive integer and $s$ be the greatest integer such that $sn \le i - 1.$ Then
	\begin{align*}\smallskip
	\ell \left(\frac{R}{c(v)_i}\right) 
	&\ge \ell \left(\frac{R}{c(v)_{sn+1}}\right) \\
	&\ge \ell \left(\frac{R}{c(v)_1}\right) 
	+ \ell \left(\frac{c(v)_n}{c(v)_{n+1}}\right) + \cdots 
	+ \ell \left(\frac{c(v)_{sn}}{c(v)_{sn+1}}\right) \\ \smallskip
	& \ge 1+2+\cdots+(s+1) = \frac{(s+1)(s+2)}{2} \ge \frac{i^2}{2n^2}
	\end{align*}
	as $(s+1)n \ge i.$ The result follows upon taking $\lambda=1/2n^2.$
\end{proof}

\begin{Theorem} \label{alternatedim2}
	Let $(R,\mathfrak{m})$ be a $2$-dimensional quasi-unmixed local ring with	infinite residue field. Let $I,J$ be $\mathfrak{m}$-primary ideals of $R.$ Suppose $r$ and $s$ are positive integers such that $r^2 e(I) = s^2 e(J) = rs e_1(I|J).$ Then $I^r$ and $J^s$ have the same integral closure. 
\end{Theorem}

\begin{proof} 
	We may assume that $R$ is a complete local ring with infinite residue field. It is sufficient to show that if $e(I)=e(J)=e_1(I|J),$ then $\overline{I} = \overline{J}.$ We may assume that $I=(a_1,a_2)$ is generated by system of parameters. Using Lemma \ref{Lengthlemma}
	\begin{align*}
	\ell\left(\frac{R}{J^n : (a_1^n,a_2^n)}\right) \le \ell\left(\frac{R}{(a_1^n,a_2^n)}\right) + 2\ell\left(\frac{R}{J^n}\right) - \ell\left(\frac{R}{I^nJ^n}\right).
	\end{align*}	
	Taking limit as $n \rightarrow \infty$ after dividing by $n^2/2$ and using Lech's formula, it follows that 
	\begin{align*}
	\lim\limits_{n \rightarrow \infty} \ell\left(\frac{R}{J^n : (a_1^n,a_2^n)}\right) \cdot \frac{1}{n^2} =0.
	\end{align*}
	We show that both $a_1$ and $a_2$ are integral over $J.$ Suppose $a_1$ is not integral over $J.$ Then there exists a minimal prime $\mathfrak{p}$ of $R$ so that $a_1$ is not integral over $J+\mathfrak{p}/\mathfrak{p}.$ Now $R/\mathfrak{p}$ is a complete local domain of dimension $2.$ Hence by Corollary \ref{val}, there exists an $\mathfrak{m}$-valuation $v$ of $R/\mathfrak{p}$ so that $v(a_1) < v(J).$ Set $j = v(J)-v(a_1) >0.$ Let $b \in J^n : (a_1^n,a_2^n)$, then $ba_1^n \in J^n$ and hence $v(b)+nv(a_1) \ge nv(J).$ This implies that $v(b) \ge nj$ and using the notation as in Lemma \ref{1}, we get $b \in c(v)_{nj}.$ Therefore for all $n >0,$ $(J^n : (a_1^n,a_2^n)) \subseteq c(v)_{nj}$ and hence using Lemma \ref{1},
	\begin{align*}
	\ell\left(\frac{R}{J^n : (a_1^n,a_2^n)}\right) \ge \ell\left(\frac{R}{c(v)_{nj}}\right) \ge \lambda(nj)^2 >0.
	\end{align*} 
	Thus
	\begin{align*}
	\lim\limits_{n \rightarrow \infty} \ell\left(\frac{R}{J^n : (a_1^n,a_2^n)}\right) \cdot \frac{1}{n^2} \ge \lambda j^2 >0
	\end{align*}
	giving a contradiction. So both $a_1$ and $a_2$ are integral over $J$ and hence $I \subseteq \overline{J}.$ By symmetry, $J \subseteq \overline{I}$, completing the proof.
\end{proof}

\section{ Minkowski's equality in dimension $d \geq 3$ }
The proof of Minkowski equality in dimension $3$ or higher, was reduced to dimension $2$ by D. Katz \cite{katz1988} by passing to certain subrings of the total quotient ring of $R$ having smaller dimension than $\dim R.$  In this section, we present an alternative proof using the specialization property of the integral  closure of an ideal first proved by  Shiroh Itoh \cite[Theorem 1]{itoh1992} if $R$ is Cohen-Macaulay of dimension $d\geq 2$ and is analytically unramified. Hong and Ulrich \cite{hongUlrich} provided another proof for analytically unramified universally catenary local rings.

We use the notion of general extensions in the proof. Let $x_1,\ldots,x_n$ be a set of indeterminates over $R.$ The ring $S=R[x_1,\ldots,x_n]_{\mathfrak{m} R[x_1,\ldots,x_n]}$ is called a \emph{general extension} of the ring $R.$ Let $I=(a_1,\ldots,a_n)$ be an ideal of $R.$ The element $x = \sum_{i=1}^{n}a_ix_i$ is called a \emph{general element} of $I.$ Then $S/R$ is a faithfully flat extension, $\dim R= \dim S$ and $e(J,R) = e(JS,S)$ for any $\mathfrak{m}$-primary ideal $J$ of $R.$ For proof of these statements and other properties of the ring $S$, the reader may refer to \cite{rees1986}.

\begin{Theorem} [{\rm{\cite[Lemma 2.4]{rees1986}}}] \label{general}
	Let $(R,\mathfrak{m})$ be a Noetherian local ring and $I=(a_1,\ldots,a_n), J$ be ideals of $R.$ Let $S=R[x_1,\ldots,x_n]_{\mathfrak{m} R[x_1,\ldots,x_n]}$ be a general extension of $R.$ Then $x=a_1x_1+\cdots+a_nx_n$ is Rees superficial for the pair $(IS,JS).$
\end{Theorem}

We are now ready to give the proof of the main theorem.

\begin{Theorem} \label{equality}
	Let $(R,\mathfrak{m})$ be a $d$-dimensional quasi-unmixed local ring with infinite residue field and $d \ge 2.$ Let $I,J$ be $\mathfrak{m}$-primary ideals. Then the following are equivalent:\\
	{\rm (1)} $e(IJ)^{1/d} = e(I)^{1/d} + e(J)^{1/d}.$\\
	{\rm (2)} There exists positive integers $r,s$ such that 
	\[ \frac{e_1(I|J)}{e_0(I|J)} = \frac{e_2(I|J)}{e_1(I|J)} = \cdots = \frac{e_d(I|J)}{e_{d-1}(I|J)} = \frac{r}{s}.\]
	{\rm (3)} There exists positive integers $r,s$ such that $\overline{I^r} = \overline{J^s}.$ 
\end{Theorem}

\begin{proof}
	\textbf{(1) $\Rightarrow$ (2)}: Consider 
	\[e(IJ) = (e(I)^{1/d} + e(J)^{1/d})^d = \sum_{i=0}^d \binom{d}{i} e(I)^{(d-i)/d} e(J)^{i/d}.\] 
	Since $e(IJ) = \sum_{i=0}^{d} \binom{d}{i} e_i,$ we get 
	\[ \sum_{i=0}^{d} \binom{d}{i} \left(e(I)^{(d-i)/d} e(J)^{i/d} - e_i \right) = 0.\] 
	Using  the inequality $e_i^d\leq e(I)^{d-i}e(J)^i,$ it follows that $e_i = e(I)^{(d-i)/d} e(J)^{i/d}$ for all $i=0,\ldots,d.$ For $i=1,\ldots, d,$ consider
	\[\frac{e_i}{e_{i-1}} = \frac{e(I)^{(d-i)/d} e(J)^{i/d}}{e(I)^{(d-i+1)/d} e(J)^{(i-1)/d}} = \left( \frac{e(J)}{e(I)} \right)^{1/d}.\]
	Hence $\frac{e_i}{e_{i-1}}=\frac{e_1}{e_0}=\frac{r}{s}$, for all $i=1,2\ldots,d.$
	\smallskip
	
	\textbf{(2) $\Rightarrow$ (1)}: Note that it is sufficient to show that $e_i = e(I)^{(d-i)/d} e(J)^{i/d}$, for all $i=0,1,\ldots,d.$ Observe that 
	\begin{align*}
	\left(\frac{e_i}{e_0} \right)^{d-i}
	&= \left(\frac{e_i}{e_{i-1}}\right)^{d-i} \cdots \left(\frac{e_1}{e_0}\right)^{d-i} \\
	&= \left(\frac{e_d}{e_{d-1}} \cdots \frac{e_{i+1}}{e_i} \right) \cdots \left(\frac{e_d}{e_{d-1}} \cdots \frac{e_{i+1}}{e_i} \right) 
	= \left(\frac{e_d}{e_i}\right)^i.
	\end{align*}
	To prove the equivalence of (2) and (3) it is sufficient to show that $e_0=e_1=\cdots=e_d$ if and only if $\overline{I} = \overline{J}.$ As $e_i(I^r|J^s) = r^{d-i}s^i e_i(I|J),$ we may replace $I^r$ by $I$ and $J^s$ by $J.$ 
	\smallskip
	
	\textbf{(3) $\Rightarrow$ (2)} : We have proved it as Proposition \ref{equal}.\\
	Observe that the quasi-unmixed property of the ring was not used in proving the above equivalences. We however need it to prove the following:
	\smallskip
	
	\textbf{(2) $\Rightarrow$ (3)} : We proceed by induction on $d.$ The case $d=2$ is proved in Theorem \ref{alternatedim2}. Suppose $d>2.$ 
	
	We may assume that $R$ is complete and $I=(a_1,\ldots,a_d)$ is generated by system of parameters. Let $Z_1,\ldots,Z_d$ be indeterminates over $R$ and put $S=R[Z_1,\ldots,Z_d]_{\mathfrak{m} R[Z_1,\ldots,Z_d]}.$ Let $x = \sum_{i=1}^{d} a_iZ_i \in S.$ Using Theorem \ref{general}, it follows that there exists $c>0$ such that for all $r \ge c$ and $s \ge 0$,
	\begin{align*}
	I^rJ^rS \cap (x)S = xI^{r-1}J^sS.
	\end{align*}
	We now prove that $S/(x)$ is quasi-unmixed. By \cite[Theorem 31.6]{matsumuraCRT}, $R$ is universally catenary and hence $S$ is catenary. Since $\dim S/(x)=d-1,$ any minimal prime of $(x)$ in $S$ has height $1.$  By the catenary property of $S$, it follows that $\dim(S/\mathfrak{p}) = \dim(S/(x))=d-1$ for all minimal primes $\mathfrak{p}$ of $(x)$ in $S.$ Therefore $S/(x)$ is  quasi-unmixed by \cite[Theorem 31.6]{matsumuraCRT}. Since $(0:_Sx) \subseteq (0:_S IS),$ using Theorem \ref{reesSup} it follows that
	\begin{align*}
	\ell \left(\frac{S}{I^rJ^sS}\right) - \ell \left(\frac{S}{I^{r-1}J^sS}\right) = \ell \left(\frac{S}{I^rJ^sS+(x)S}\right) - \ell((0 :_S x)).
	\end{align*}
	Let $e'_0,\ldots,e'_d$ denote the mixed multiplicities of $I$ and $J$ as ideals in $S/(x).$ Then using above equation, $e_0=e'_0, \ldots, e_{d-1}=e'_{d-1}.$ As $\dim S/(x)S = d-1,$ using induction hypothesis we get $J(S/(x)) \subseteq \overline{I(S/(x))}.$ Let $\mathfrak{p}$ be a minimal prime of $R$ and $\tilde{R} = R/\mathfrak{p}.$ Consider the natural surjective map
	\begin{align*}
	\frac{S}{(x)} = \frac{R[Z_1,\ldots,Z_d]_{\mathfrak{m} R[z_1,\ldots,Z_d]}}{(x)} \rightarrow T = \frac{\tilde{R}[Z_1,\ldots,Z_d]_{\mathfrak{m}\tilde{R}[z_1,\ldots,Z_d]}}{(x)} \rightarrow 0.
	\end{align*}
	Using the persistence property of integral closure, it follows that $JT \subseteq \overline{IT}.$ As $\tilde{R}$ is a complete local domain, it is analytically unramified and hence $JT \subseteq \overline{IT} = \overline{I}T$ by \cite[Corollary 2.3]{hongUlrich}. Therefore, 
	\begin{align*}
	J\tilde{R} \subseteq JT \cap \tilde{R} \subseteq \overline{I\tilde{R}}.
	\end{align*}
	Since this is true for all minimal prime ideals $\mathfrak{p}$ of $R$, it follows that $J \subseteq \overline{I}.$ By symmetry $I \subseteq \overline{J}$, completing the proof.
\end{proof}

We now present an algebraic proof of the Rees multiplicity theorem using Minkowski's equality and a geometric proof in dimension $2$ using negative definiteness of the intersection form. The algebraic proof is given by Teissier \cite{teissier1978}.

\begin{Theorem}[{\bf Rees}, \cite{reesmul}] \label{Rees}
	Let $(R,\mathfrak{m})$ be a $d$-dimensional quasi-unmixed local ring and $J\subseteq I$ be $\mathfrak{m}$-primary ideals. If $e(I)=e(J)$, then $\overline{I}=\overline{J}.$
\end{Theorem}

\begin{proof} 
	Note that $e(IJ)\geq e(I^2)=2^d e(I).$ On the other hand,
	\[e(IJ)= \sum_{i=0}^d\binom{d}{i}e_i(I|J)\leq  \sum_{i=0}^d\binom{d}{i}e(I)^{\frac{d-i}{d}}e(J)^{\frac{i}{d}}=2 ^de(I).\]
	Hence  $e_i(I|J)=e(I)$ for all $i=0,\ldots,d.$ Since $R$ is quasi-unmixed, it follows that $\overline{I}=\overline{J}.$ 
\end{proof}

\begin{Theorem}
	Let $R$ be a geometric local domain of dimension $2$, $I \subset J$ be ideals primary for the maximal ideal of $R.$ If $e(I)=e(J),$ then $\overline{I} = \overline{J}.$
\end{Theorem}

\begin{proof}
	Let $f : X \to \Spec R$ be a resolution of singularities such that $f^*I, f^*J$ are locally invertible sheaves of ideals in $X.$ They define divisors $D,E$ in $X$ such that $E \leq D.$ By \cite{CPR}, $e(I)=-D^2$, $e(J)=-E^2.$
	
	Write $D=E+E'$ for an effective divisor $E'.$ One property of $D$ and $E$ is that $D.D_i \leq 0$ for every irreducible component $D_i$ of support $D$ with strict inequality for some $i$ (similarly for $E$), and every $D_i$ occurs in both $D,E.$
	
	Now $(E+E')^2=E^2+2E.E'+E'^2.$ The term $E.E' \leq 0$ by the remark above. If $E'$ is non-zero, then $E'^2 < 0$ by negative definiteness of the intersection form. It follows that if $D$ is strictly bigger than $E$, then $D^2 < E^2$, i.e., $e(I) > e(J).$ By assumption, $e(I)=e(J).$ Hence $D=E.$ Since $\overline I\subset \overline J$, the proof shows that $\overline{I} = \overline{J}.$
\end{proof}

We end with an example which illustrates that Theorem \ref{equality} fails if the ring is not quasi-unmixed.

\begin{Example}{\rm 
	Let $R=k[[X,Y,Z]]/(XY,XZ)$ where $k$ is an infinite field and let $x,y,z$ denote the images of $X,Y$ and $Z$ in $R$ respectively. Observe that $R$ is a 2-dimensional Noetherian local ring but it is not quasi-unmixed as the minimal primes of $R$ are $(x)$ and $(y,z).$ Set $\mathfrak{m}=(x,y,z)$ and  $I=(y,x^2+z)$. We first show that $e(I) = e(\mathfrak{m}).$ Using the associativity formula, we get 
	\[ e(I,R) = e\left(\frac{I+(x)}{(x)},\frac{R}{(x)}\right) \ell(R_{(x)}) = e((y,z),k[[y,z]]) =1 \]
	and
	\[ e(\mathfrak{m},R) = e\left(\frac{\mathfrak{m}+(x)}{(x)},\frac{R}{(x)}\right) \ell(R_{(x)}) = e((y,z),k[[y,z]]) =1. \]
	As in the proof of Theorem \ref{Rees}, one can now conclude that $e(I) = e_1(I|\mathfrak{m}) = e(\mathfrak{m}).$ 
	
	We claim that $\overline{I} \neq \overline{\mathfrak{m}}.$ As $z$ satisfies the equation $z^2-z(x^2+z)=0$, it follows that $z \in \overline{I}$ and hence $(x^2,y,z) \subseteq \overline{I}.$ In order to prove the claim, it is now sufficient to show that $x \notin \overline{I}.$ If $x$ is integral over $I$, then it is integral over $I(R/(y,z))$ in the ring $R/(y,z).$ In other words, $x$ is integral over $(x^2)$ in $k[[x]].$ This gives a contradiction as $(x^2)$ is integrally closed in $k[[x]].$ Hence the claim is true.}
\end{Example}

\bibliographystyle{amsplain}

\end{document}